\documentclass[12pt, reqno]{amsart}
\usepackage{amssymb,latexsym,amsmath,amscd,amsthm,graphicx, color}
\usepackage[all]{xy}
\usepackage{pgf,tikz}
\usepackage{mathrsfs}
\usetikzlibrary{arrows}
\usepackage[left=0.6 in, top=0.6 in, right=0.6 in, bottom=0.3 in]{geometry}

\allowdisplaybreaks
\makeatletter
\newcommand{\setword}[2]{%
	\phantomsection
	#1\def\@currentlabel{\unexpanded{#1}}\label{#2}%
}
\makeatother

\usepackage[linkcolor=blue, urlcolor=blue, citecolor=blue,
colorlinks, bookmarks]{hyperref}

\definecolor{uuuuuu}{rgb}{0.26666666666666666,0.26666666666666666,0.26666666666666666}
\definecolor{xdxdff}{rgb}{0.49019607843137253,0.49019607843137253,1.}
\definecolor{ffqqqq}{rgb}{1.,0.,0.}
\definecolor{ffqqqq}{rgb}{1.,0.,0.}
\definecolor{ffxfqq}{rgb}{1.,0.4980392156862745,0.}

\definecolor{uuuuuu}{rgb}{0.26666666666666666,0.26666666666666666,0.26666666666666666}
\definecolor{qqwuqq}{rgb}{0.,0.39215686274509803,0.}
\definecolor{zzttqq}{rgb}{0.6,0.2,0.}
\definecolor{xdxdff}{rgb}{0.49019607843137253,0.49019607843137253,1.}
\definecolor{qqqqff}{rgb}{0.,0.,1.}
\definecolor{cqcqcq}{rgb}{0.7529411764705882,0.7529411764705882,0.7529411764705882}
\definecolor{sqsqsq}{rgb}{0.12549019607843137,0.12549019607843137,0.12549019607843137}
\definecolor{ududff}{rgb}{0.30196078431372547,0.30196078431372547,1.}

\theoremstyle{plain}

\newtheorem{lemma}[subsection]{Lemma}
\newtheorem{defi}[subsection]{Definition}
\newtheorem{prop}[subsection]{Proposition}

\newtheorem{maintheorem}[subsection]{Main Theorem}

\theoremstyle{definition}
\newtheorem{defi1}[subsection]{Definition}

\newtheorem{remark}[subsection]{Remark}

\newtheorem{note}[subsection]{Note}

\newcommand{\uu}{\cup}
\newcommand{\ii}{\cap}

\newcommand{\set}[1]{\{#1\}}

\newcommand{\ga}{\alpha}

\newcommand{\tbf}{\textbf}
\newcommand{\tit}{\textit}

\newcommand{\D}[1]{\mathbb{#1}}

\newcommand{\te}{\text}

\begin{document}

To appear, Houston Journal of Mathematics 
\title{Quantization for the mixtures of overlap probability distributions}

\address{School of Mathematical and Statistical Sciences\\
University of Texas Rio Grande Valley\\
1201 West University Drive\\
Edinburg, TX 78539-2999, USA.}

\email{$^1$ashabarua@vt.edu}
\email{\{$^{2}$angelinachavera1, $^3$ivan.djordjevic.us, $^4$vmdance94, $^7$sophia.tejeda.06\}@gmail.com}
\email{$^{5}$ssoto14@stedwards.edu}
\email{$^6$mrinal.roychowdhury@utrgv.edu}

 \author{$^1$Asha Barua}
\author{$^2$Angelina Chavera}
\author{$^3$Ivan Djordjevic}
\author{$^4$Valerie Manzano}
\author{$^5$Sergio Soto Quintero}
 \author{$^6$Mrinal Kanti Roychowdhury}
 \author{$^7$Hilda Tejeda}


\date{}
\maketitle

\pagestyle{myheadings}\markboth{Barua, Chavera, Djordjevic, Manzano, Quintero, Roychowdhury, and Tejeda}{Quantization for the mixtures of overlap probability distributions}

\begin{abstract}
Optimal quantization for mixed distributions has emerged as a compelling area of study.  In this work, we have focused on a mixed distribution formed from two uniform distributions with partially overlapping supports. For this class of distributions, we have examined the structure of optimal sets of \( n \)-means and the corresponding \( n \)th quantization errors for all positive integers \( n \). Initially, we explicitly determined the optimal sets and quantization errors for \( 1 \leq n \leq 6 \). Subsequently, we established several key lemmas and propositions and proposed an algorithm that facilitates the computation of optimal \( n \)-means and quantization errors for all \( n \geq 5 \). Numerical results are also presented to illustrate the application of the algorithm in deriving these quantities. The findings of this study offer valuable insight and serve as a foundation for further research on quantization in the context of mixed distributions with overlapping supports.

\end{abstract}

\section{Introduction}

Quantization is a nonlinear, memoryless process that transforms a continuous signal into a discrete one, restricted to a finite set of values. This process naturally arises whenever continuous physical quantities are represented in numerical form. The first systematic study of quantization is credited to W.F. Sheppard (see \cite{S}). Quantization plays a fundamental role in a wide range of fields, including signal processing, telecommunications, data compression, image analysis, and cluster analysis. For more detailed discussions and extensive bibliographies, the reader is referred to \cite{GG, GKL, GN, Z}.  Recently, Pandey and Roychowdhury introduced the concepts of constrained quantization and the conditional quantization (see \cite{PR1, PR2, PR3}). A quantization without a constraint is known as an unconstrained quantization, which traditionally in the literature is known as quantization. For a mathematical treatment of unconstrained quantization interested researchers can consult \cite{GL}. 
This paper deals with unconstrained quantization for mixed distribution. For more results about the unconstrained quantization for mixed distributions one can see \cite{R6, RS}.    

Let $\D R^d$ denote the $d$-dimensional Euclidean space endowed with a norm $\|\cdot\|$ that is compatible with the standard Euclidean topology. Consider two Borel probability measures $P_1$ and $P_2$ defined on $\D R^d$. A Borel probability measure $P$ on $\D R^d$ is referred to as a \textit{mixture} (or \textit{mixed distribution}) of $P_1$ and $P_2$, associated with the probability vector $(p, 1 - p)$, if it is given by
\[
P := p P_1 + (1 - p) P_2,
\]
where $0 < p < 1$.

The \textit{$n$th quantization error} of the measure $P$, with respect to the squared Euclidean norm, is defined as
\[
V_n := V_n(P) = \inf \left\{ V(P; \alpha) : \alpha \subset \D R^d,\; 1 \leq \te{card}(\alpha) \leq n \right\},
\]
where
\[
V(P; \alpha) = \int \min_{a \in \alpha} \|x - a\|^2 \, dP(x)
\]
denotes the \textit{distortion error} corresponding to the set $\alpha$ under the probability measure $P$.

A set $\alpha \subset \D R^d$ is said to be an \textit{optimal set of $n$-points} for $P$ if it achieves the minimum quantization error, i.e., $V_n(P) = V(P; \alpha)$. It is known  that if $P$ is a Borel probability measure with support containing at least $n$ points  and satisfies the finite second moment condition $\int \|x\|^2 \, dP(x) < \infty$, then every optimal set of $n$-points consists of exactly $n$ elements (see \cite{PR1, GL}). 

We now present a fundamental result concerning the structure of optimal sets of $n$-points (see \cite{GG, GL}):

\begin{prop} \label{prop0}
Let $\alpha$ be an optimal set of $n$-points for the probability measure $P$, and let $a \in \alpha$. Then the following properties hold:
\begin{itemize}
    \item[(i)] $P(M(a \mid \alpha)) > 0$,
    \item[(ii)] $P(\partial M(a \mid \alpha)) = 0$,
    \item[(iii)] $a = \mathbb{E}(X \mid X \in M(a \mid \alpha))$,
\end{itemize}
where $M(a|\ga)$ is the Voronoi region of $a\in \ga, $ i.e.,  $M(a|\ga)$ is the set of all elements $x$ in $\D R^d$ which are closest to $a$ among all the elements in $\ga$, and $\partial M(a|\ga)$ represents the boundary of the Voronoi region $M(a|\ga)$.
\end{prop}

By the above proposition, we see that in unconstrained quantization, the elements in an optimal set of $n$-points are the conditional expectations in their own Voronoi regions. Because of this fact, in unconstrained quantization, an optimal set of $n$-points is termed an \tit{optimal set of $n$-means}.
\begin{defi1}
Let $P$ be a Borel probability measure on $\D R^k$, and $U$ be the largest open subset of $\D R^k$ such that $P(U)=0$. Then, $\D R^k\setminus U$ is called the \tit{support of $P$}, and is denoted by $\te{supp}(P)$. \tit{Probability distributions have some overlaps} or by  \tit{overlapping probability distributions} it is meant that the underlying probability distributions have some nonempty intersection in their supports. 
\end{defi1}

In this paper we prove the following theorem, which is the main theorem of the paper. 

\begin{maintheorem}\label{main}
Let \( P_1 \) and \( P_2 \) are two uniform probability distributions on the intervals \( [0, 1] \) and \( \left[\frac{1}{2}, \frac{3}{2}\right] \), respectively. 
Let \( P:= p_1P_1 +p_2P_2 \) be the mixed distribution generated by $P_1$ and $P_2$ associated with the probability vector $(p_1, p_2)$. Take $p=\frac 12$. Then, for each integer \( n \geq 5 \), there exists an optimal set \( \alpha_n \subset \mathbb{R} \) of \( n \)-means and a corresponding \( n \)th quantization error \( V_n \) for the probability measure \( P \), such that:

\begin{itemize}
    \item[(i)] The set \( \alpha_n \) is symmetric with respect to the point \( \frac{3}{4} \) and contains elements from both intervals \( \left(0, \frac{1}{2}\right) \) and \( \left(\frac{1}{2}, \frac{3}{4}\right) \) (see Remark~\ref{rem111} and Lemma~\ref{lemma3}).

    \item[(ii)] If \( n \) is even, then \( n = 2(k + m) \); if \( n \) is odd, then \( n = 2(k + m) + 1 \), where \( k, m \in \mathbb{N} \) denote the number of points in \( \alpha_n \cap \left[0, \frac{1}{2}\right] \) and \( \alpha_n \cap \left(\frac{1}{2}, \frac{3}{4}\right) \), respectively.

    \item[(iii)] The structure of \( \alpha_n \) and the value of \( V_n \) depend on whether the midpoint \( \frac{1}{2}(a_k + b_1) \) lies to the left or right of \( \frac{1}{2} \), where \( a_k \in \alpha_n \cap \left[0, \frac{1}{2}\right] \) and \( b_1 \in \alpha_n \cap \left(\frac{1}{2}, \frac{3}{4}\right) \).

    \item[(iv)] Explicit formulas for the elements of \( \alpha_n \) and the value of \( V_n \) are provided in Propositions~\ref{prop71} to \ref{prop73}, depending on the relative position of \( \frac{1}{2}(a_k + b_1) \).

    \item[(v)] There exists a deterministic algorithm (see Subsection~\ref{Algo}) that, given \( n \geq 5 \), computes the correct values of \( k \) and \( m \), and thus enables the exact construction of \( \alpha_n \) and computation of \( V_n \).
\end{itemize}
\end{maintheorem}
 
\subsection{Application of Mixed Distribution}
 Mixed distributions represent a promising and evolving area of research in the theory of optimal quantization. In this paper, we consider a mixed distribution \( P := pP_1 + (1 - p)P_2 \), where \( 0 < p < 1 \), formed from two component probability measures \( P_1 \) and \( P_2 \) whose supports have overlap. However, the framework can be naturally extended to cases where the supports of \( P_1 \) and \( P_2 \) are disjoint.

Optimal quantization of mixed distributions finds relevance in a variety of applied fields. One such application arises in agriculture, specifically in resource optimization for irrigation. Consider a scenario where a cropland is partitioned into two regions, with the first region requiring \( k \) times more water than the second, for some positive integer \( k \). Let \( P_1 \) and \( P_2 \) denote uniform probability distributions over the first and second regions, respectively. Then, the overall distribution of water usage across the entire cropland can be modeled as a mixed distribution \( P = \frac{k}{k+1} P_1 + \frac{1}{k+1} P_2 \).

In this context, describing \( P \) as a ``uniform distribution on the land with respect to water distribution'' implies that the cropland, when divided into equal-area segments, would assign the same probability (and thus the same water allocation) to each segment under the distribution \( P \). Such an approach provides a mathematically rigorous basis for optimal placement of a minimal number of water sprinklers (or other resources), ensuring efficient coverage and distribution.

We believe that this line of investigation opens pathways to deeper insights in statistical modeling and real-world optimization, with potential for further theoretical and practical development.
\subsection{Delineation}
The organization of the paper is as follows. In Section~\ref{sec1}, we present the necessary preliminaries for a general mixed distribution \( P := pP_1 + (1 - p)P_2 \), and establish Proposition~\ref{prop55} and Proposition~\ref{prop56}, which are the key results required for the proof of the main theorem, Theorem~\ref{main}. Section~\ref{sec2} is devoted to the computation of optimal sets of \( n \)-means and the corresponding \( n \)th quantization errors for \( 1 \leq n \leq 6 \). This section also includes Lemma~\ref{lemma3}, which asserts that if \( \alpha_n \) is an optimal set of \( n \)-means for \( n \geq 4 \), then \( \alpha_n \) must contain points from both open intervals \( (0, \frac{1}{2}) \) and \( (\frac{1}{2}, \frac{3}{4}) \). In Section~\ref{sec3}, we provide general formulas and techniques for determining the optimal sets of \( n \)-means and the corresponding quantization errors for all \( n \geq 5 \). Finally, Section~\ref{sec5} concludes the paper and outlines directions for future research.

\section{basic preliminaries} \label{sec1} 

 Let $f_1$ and $f_2$ be the respective density functions for the uniform distributions $P_1$ and $P_2$ defined on the closed intervals $[0, 1]$ and $[\frac 12, \frac 32]$. Then,
 \[f_1(x)=\left \{\begin{array} {cc}
1 & \te{ if } x\in [0, 1], \\
0 & \te{otherwise};
\end{array}
\right. \te{ and } f_2(x)=\left \{\begin{array} {cc}
1 & \te{ if } x\in [\frac 12, \frac 32], \\
0 & \te{otherwise}.
\end{array}
\right.\]
 Let us consider the mixed distribution $P:=pP_1+(1-p)P_2$, where $0<p<1$. Notice that $P$ has support the closed interval $[0, \frac 32]$, and the component probabilities $P_1$ and $P_2$ have overlaps on the interval $[\frac 12, 1]$. For a probability distribution $P$, by $dP(x)$ it is meant $dP(x)=P(dx)$, where $d$ stands here for differential. Since $f_1$ and $f_2$ are the density functions for the probability distributions $P_1$ and $P_2$, respectively, we have
 \[dP_1(x)=P_1(dx)=f_1(x)dx=dx,   \te{ and } dP_2(x)=P_2(dx)=f_2(x) dx=dx.\]
Notice that if $x\in [0, \frac 12]$, then  $dP(x)=p \,dP_1(x)=p dx$; if $x\in [\frac 12,1]$, then
$dP(x)=p  dP_1(x)+(1-p) dP_2(x)=p f_1(x) dx+(1-p)f_2(x) dx=dx;$
on the other hand, if $x\in [1, \frac 32]$, then $dP(x)=(1-p)dP_2(x)=(1-p)dx$. Let us now define a function $f$ on the real line $\D R$ such that
\begin{align} \label{eq1} f(x)=\left\{\begin{array}{cc}
p & \te{ if } x \in [0, \frac 12], \\
1  & \te{ if } x \in [\frac 12, 1],\\
(1-p) &  \te{ if } x \in [1, \frac 32],\\
0 & \te{otherwise}.
\end{array}
\right.
\end{align}
Notice that the function $f$ satisfies the following properties to be a density function:
\[f(x)\geq 0 \te{ for all } x\in \D R, \te{ and } \int_{-\infty}^\infty f(x) dx=1.\]
Indeed, the mixed distribution $P$ can now be identified as a probability distribution on $\D R$ with the density function $f$, i.e., for any $x\in \D R$, we have $dP(x)=P(dx)=f(x)dx$. 

The notations \( E(X) \) and \( V(X) \) denote the expected value and variance, respectively, of a random variable \( X \) with respect to a probability distribution \( P \). The expected value \( E(X) \) corresponds to the mean or average of the distribution, reflecting the central tendency of \( X \). The variance \( V(X) \), on the other hand, quantifies the degree of dispersion or spread of the values of \( X \) around its mean. Formally, the variance is defined as the expected value of the squared deviation from the mean, i.e.,
\[
V(X) = E[(X - E[X])^2],
\]
which can equivalently be expressed as
\[
V(X) = E(X^2) - (E(X))^2.
\]

\begin{lemma} \label{lemma0}
Let \( X \) be a random variable distributed according to \( P \). Then, the expected value and variance of \( X \) are given by
\[
E(X) = \frac{2 - p}{2} \quad \text{and} \quad V(X) = \frac{1}{12}(-3p^2 + 3p + 1),
\]
respectively.

\end{lemma}

\begin{proof}
We have
\[E(X)=\int x dP=p \int_0^{\frac{1}{2}} x \, dx+\int_{\frac{1}{2}}^1 x \, dx+(1-p) \int_1^{\frac{3}{2}} x \, dx=\frac{2-p}{2},\]
and
\begin{align*} & V(X)=\int (x-E(X))^2 dP\\
&=p \int_0^{\frac{1}{2}} (x-\frac{2-p}{2})^2 \, dx+\int_{\frac{1}{2}}^1 (x-\frac{2-p}{2})^2 \, dx+(1-p) \int_1^{\frac{3}{2}} (x-\frac{2-p}{2})^2 \, dx,
\end{align*}
implying
$V(X)=\frac{1}{12} (-3 p^2+3 p+1)$, and thus, the lemma is yielded.
\end{proof}

\begin{note} Lemma~\ref{lemma0} implies that the optimal set of one-mean is the set $\set{\frac{2-p}{2}}$, and the corresponding quantization error is the variance $V:=V(X)$ of a random variable with distribution $P$. For a subset $J$ of $\D R$ with $P(J)>0$,
by $P(\cdot|_J)$, we denote the conditional probability given that $J$ is occurred, i.e., $P(\cdot|_J)=P(\cdot\ii J)/P(J)$, in other words, for any Borel subset $B$ of $\D R$ we have $P(B|_J)=\frac{P(B\ii J)}{P(J)}$.
\end{note}

\begin{prop} \label{prop55}
Let $P$ be a Borel probability measure on $\D R$ such that $P$ is uniformly distributed over a closed interval $[a, b]$ with a constant density function $f$ such that $f(x)=t$ for all $x\in [a, b]$, where $t\in \D R$. Then, the optimal set $\ga_n(P(\cdot|_{[a, b]}))$ of $n$-means and the corresponding quantization error $V_n(P(\cdot|_{[a,b]}))$ of $n$-means for the probability distribution $P(\cdot|_{[a, b]})$ are, respectively, given by
\[\ga_n(P(\cdot|_{[a, b]})):=\Big\{a+\frac{(2j-1)(b-a)}{2n} : 1\leq j\leq n\Big\} \te{ and } V_n(P(\cdot|_{[a,b]})):=\frac {(b-a)^3t}{12n^2}.\]
\end{prop}
\begin{proof}
Let $\ga_n:=\set{a_1<a_2<\cdots<a_n}$ be an optimal set of $n$-means for the probability distribution $P$ with a constant density function $f$ on $[a, b]$ such that $f(x)=t$ for all $x\in [a, b]$, where $t\in \D R$. Then, proceedings analogously as \cite[Theorem~2.1.1]{RR}, we can show that  $a_j=a+\frac{(2j-1)(b-a)}{2n}$ implying
\[\ga_n(P(\cdot|_{[a, b]}))=\Big\{a+\frac{(2j-1)(b-a)}{2n} : 1\leq j\leq n\Big\}.\]
 Notice that the probability density function is constant, and the Voronoi regions of the elements $a_j$ for $1\leq j\leq n$ are of equal lengths. This yields the fact that the distortion errors due to each $a_j$ are equal. Hence, the $n$th quantization error is given by
\begin{align*}
V_n(P(\cdot|_{[a, b]}))&=\int \min_{a\in \ga_n(P(\cdot|_{[a, b]}))} (x-a)^2dP=n t\int_a^{\frac 12(a_1+a_2)}  (x-a_1)^2  dx=nt \int_{a}^{a+\frac {b-a}{n}}(x-(a+\frac{b-a}{2n}))^2 dx
\end{align*}
implying
\[V_n(P(\cdot|_{[a, b]}))=\frac {(b-a)^3t}{12n^2}.\]
 Thus, the proof of the proposition is complete.
\end{proof}

\begin{remark}
If there is a restriction on the elements in an optimal set of $n$-means, for example, see Proposition~\ref{prop56},  then all the elements in an optimal set of $n$-means may not be the conditional expectations of their own Voronoi regions. Still in the sequel, we will call it as an optimal set of $n$-means, and apologize for the abuse of terminology.   
\end{remark} 
 
\begin{prop} \label{prop56}
Let $P$ be a Borel probability measure on $\D R$ such that $P$ is uniformly distributed over a closed interval $[a, b]$ with a constant density function $f$ such that $f(x)=t$ for all $x\in [a, b]$, where $t\in \D R$. Let $\ga_{n+1}(P(\cdot|_{[a, b]})):=\set{a_1<a_2<\cdots<a_n<a_{n+1}}$ be an optimal set of $(n+1)$-means such that it always contains the endpoint $b$ of the interval $[a, b]$. Then, $a_1, a_2, \cdots, a_n, a_{n+1}$ are given by 
\[a_j=\left\{\begin{array}{cc}
a+\frac{(2j-1)(b-a_1)}{2n} & \te{ for } 1\leq j\leq n,\\
b & \te{ if } j=n+1,
\end{array}
\right.
\]
and the corresponding quantization error is given by
\[V_{n+1}(P(\cdot|_{[a, b]}))=\frac{t}{12n^2}(b-a_1)^3+\frac{t}{3}(a_1-a)^3.\]
\end{prop}
\begin{proof}
Let $\ga_{n+1}(P(\cdot|_{[a, b]})):=\set{a_1<a_2<\cdots<a_n<a_{n+1}}$ be an optimal set of $(n+1)$-means such that it always contains the endpoint $b$ of the interval $[a, b]$, i.e., $a_{n+1}=b$. Let us first prove the following claim.

\tit{Claim. $a_1, a_2, \cdots, a_n$ are uniformly distributed over the closed interval $[a, \frac 12(a_n+b)]$.}

Recall that if there is no restriction, the elements in an optimal set are the conditional expectations in their own Voronoi regions. By the statement of the proposition, in the optimal set of $(n+1)$-means there is no restriction over the elements $a_1, a_2, \cdots, a_n$. Thus,
\[a_1=E(X : X \in [a, \frac 12(a_1+a_2)])=\frac{1}{4} (2 a+a_1+a_2) \te{ implying } a_2-a_1=2(a_1-a).\]
For $2\leq i\leq n$,
\[a_i=E(X : X\in [\frac 12(a_{i-1}+a_i), \frac 12(a_i+a_{i+1})])=\frac{1}{4} (a_{i-1}+2 a_i+a_{i+1}),\]
yielding $a_i-a_{i-1}=a_{i+1}-a_i$. Hence,
\begin{equation} \label{Me420}  a_2-a_1=a_3-a_2=\cdots=b-a_n=2(a_1-a).
\end{equation}  Notice that $b-a_n=2(a_1-a)$ implies that $\frac 12(a_n+b)-a_n=a_1-a$.
Thus, we deduce that $a_1, a_2, \cdots, a_n$ are uniformly distributed over the closed interval $[a, \frac 12(a_n+b)]$, which is the claim.
Due to the claim, using Proposition~\ref{prop55}, we have
\[a_j=\set{a+\frac{(2j-1)}{2n} (\frac 12(a_n+b)-a) \te{ for } 1\leq j\leq n}\]
yielding
\begin{align} \label{eq22}  a_j=\left\{\begin{array}{cc}
a+\frac{(2j-1)}{2n} (\frac 12(a_n+b)-a) & \te{ for } 1\leq j\leq n,\\
b & \te{ if } j=n+1.
\end{array}
\right.
\end{align}
By \eqref{Me420}, we have 
$\frac 12(a_n+b)-a=b-a_1.$
Hence, by \eqref{eq22}, we have
\[a_j=\left\{\begin{array}{cc}
a+\frac{(2j-1)(b-a_1)}{2n} & \te{ for } 1\leq j\leq n,\\
b & \te{ if } j=n+1.
\end{array}
\right.
\]
To find the quantization error we proceed as follows: Since the elements $a_1, a_2, \cdots, a_n$ are uniformly distributed over the closed interval $[a, \frac 12(a_n+b)]$, by Proposition~\ref{prop55}, the quantization error contributed by the elements $a_1, a_2, \cdots, a_n$ over the closed interval $[a, \frac 12(a_n+b)]$ is given by
\begin{equation} \label{eq567} \frac{(\frac 12(a_n+b)-a)^3 t}{12 n^2}=\frac{t}{12n^2}(b-a_1)^3.\end{equation}
The quantization error contributed by $b$ in the closed interval $[\frac 12(a_n+b), b]$ is given by
\begin{equation} \label{eq568}\int_{\frac 12(a_n+b)}^b(x-b)^2 dP=t\int_{b+a-a_1}^b (x-b)^2 \, dx=\frac{t}{3}  (a_1-a)^3.\end{equation}
By \eqref{eq567} and \eqref{eq568}, we have
\[V_{n+1}(P(\cdot|_{[a, b]}))=\frac{t}{12n^2}(b-a_1)^3+\frac{t}{3}(a_1-a)^3.\]
Thus, the proof of the proposition is complete.
\end{proof}

In the following two sections we give the main results of the paper taking $p=\frac 12$ in the mixed distribution $P:=pP_1+(1-p) P_2$, i.e., in the following two sections we calculate all the optimal sets of $n$-means and the $n$th quantization errors for all $n\in \D N$ for the mixed distribution $P:=\frac 12 P_1+\frac 12 P_2$.

\section{Optimal sets of $n$-means and the $n$th quantization errors for all $1\leq n\leq 6$. } \label{sec2}

Recall that the mixed distribution $P$ is identified as a probability distribution with density function $f$. If $p=\frac 1{2}$, then $E(X)=\frac{3}{4}$ and $V(X)=\frac{7}{48}$, i.e., the optimal set of one-mean for $p=\frac 12$ is $\set{\frac{3}{4}}$ and the corresponding quantization error is $V(X)=\frac{7}{48}$. For $p=\frac 12$, the density function $f$ for $P$ represented by \eqref{eq1} reduces to
\begin{align*} \label{eq1} f(x)=\left\{\begin{array}{cc}
\frac 12 & \te{ if } x \in [0, \frac 12]\uu [1, \frac 32], \\
1  & \te{ if } x \in [\frac 12, 1],\\
0 &  \te{ otherwise}.
\end{array}
\right.
\end{align*}
Notice that the probability measure $P$ is `symmetric' about the element $\frac 34$, i.e., if two intervals of equal lengths are
equidistant from the element $\frac 34$, then they have the same $P$-measure (see Figure~\ref{Fig1}).

\begin{remark} \label{rem111}
Since the probability distribution is symmetric about the element $\frac 34$, without any loss of generality we can always assume that if $\ga_n$ is an optimal set of $n$-means, then for an odd positive integer $n$ the element $\frac 34\in \ga_n$ and all other elements in $\ga_n$ are equally distributed on both sides of $\frac 34$; on the other hand, if $n$ is an even positive integer, then all the elements in the optimal set will be equally distributed on both sides of $\frac 34$. Thus, we see that $n$ is even or odd, in any case, an optimal set $\ga_n$ of $n$-means contains equal number of elements from both sides of the element $\frac 34$ (see Figure~\ref{Fig2}).
\end{remark}

\begin{figure}
\centerline{\includegraphics[width=7 in, height=5 in]{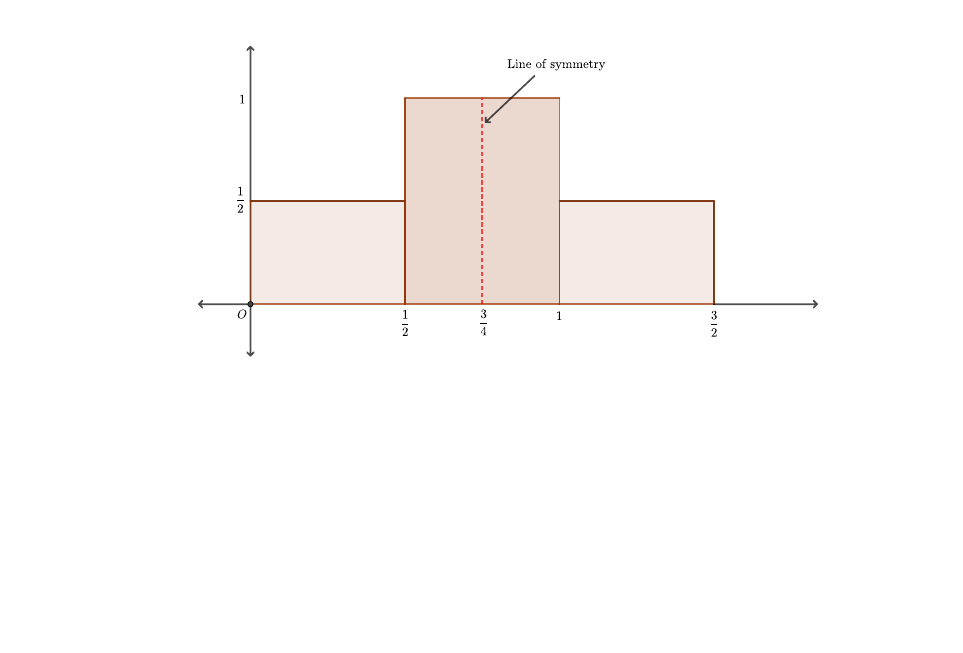}}
\vspace{-2.5 in} 
\caption{Density function $f$ for the nonuniform probability distribution $P$.}\label{Fig1}
\end{figure}

\begin{prop}   \label{prop1}
The optimal set of two-means is  $\set{\frac{7}{16}, \frac{17}{16}}$ with quantization error $V_2=\frac{37}{768}$.
\end{prop}

\begin{proof}
Let $\ga:=\set{a_1, a_2}$ be an optimal set of two-means. Since the elements in an optimal set are the conditional expectations in their own Voronoi regions, we can assume that $0<a_1<a_2<\frac 32$. Again, due to symmetry of the probability distribution $P$ about the element $\frac 34$, we can assume that the boundary $\frac 12(a_1+a_2)$ of the Voronoi regions of $a_1$ and $a_2$ passes through the midpoint $\frac 34$ of the support of $P$. Thus, we have
\[a_1=E(X : X\in [0, \frac 34])=\frac{\frac{1}{2} \int_0^{\frac{1}{2}} x \, dx+\int_{\frac{1}{2}}^{\frac{3}{4}} x \, dx}{\frac{1}{2} \int_0^{\frac{1}{2}} 1 \, dx+\int_{\frac{1}{2}}^{\frac{3}{4}} 1 \, dx}=\frac{7}{16},\]
and since $\frac 12(a_1+a_2)=\frac 34$, we have $a_2=\frac{17}{16}$. Again, due to symmetry, the quantization error for two-means is given by
\[V_2=2 \Big(\frac{1}{2} \int_0^{\frac{1}{2}} (x-\frac{7}{16})^2 \, dx+\int_{\frac{1}{2}}^{\frac{3}{4}}(x-\frac{7}{16})^2 \, dx\Big)=\frac{37}{768}.\]
Thus, the proof of the proposition is complete (also see Figure~\ref{Fig2}).
\end{proof}

\begin{prop}   \label{prop2}
The optimal set of three-means is  $\set{\frac 14, \frac 34, \frac 54}$ with quantization error $V_3=\frac{1}{48}$.
\end{prop}

\begin{proof}
Let $\ga$ be an optimal set of three-means.
As mentioned in Remark~\ref{rem111}, we can assume that $\frac 34\in \ga$. Let the other two elements in $\ga$ be $a_1$ and $a_2$ such that $0<a_1<\frac 34<a_2<\frac 32$. Now, the boundary of the Voronoi regions of $a_1$ and $\frac 34$ is $\frac 12(a_1+\frac 34)$. The following two cases can arise:

\tit{Case~1. $\frac 12(a_1+\frac 34)\leq \frac 12.$}

In this case, due to symmetry the distortion error is given by
\begin{align*}
&\int \min_{a\in \ga}(x-a)^2 dP= 2 \Big(\frac{1}{2} \int_0^{\frac{1}{2} (a_1+\frac{3}{4})} (x-a_1){}^2 \, dx+\frac{1}{2} \int_{\frac{1}{2} (a_1+\frac{3}{4})}^{\frac{1}{2}} (x-\frac{3}{4})^2 \, dx+\int_{\frac{1}{2}}^{\frac{3}{4}} (x-\frac{3}{4})^2 \, dx\Big)\\
&=\frac{1}{768} \Big(192 a_1^3+144 a_1^2-108 a_1+31\Big),
\end{align*}
the minimum value of which is $0.0208333$ and it occurs when $a_1=\frac 14$.

\tit{Case~2. $\frac 12\leq \frac 12(a_1+\frac 34).$}

In this case, due to symmetry the distortion error is given by
\begin{align*}
&\int \min_{a\in \ga}(x-a)^2 dP= 2 \Big(\frac{1}{2} \int_0^{\frac{1}{2}} (x-a_1 ){}^2 \, dx+\int_{\frac{1}{2}}^{\frac{1}{2}  (a_1+\frac{3}{4} )}  (x-a_1 ){}^2 \, dx+\int_{\frac{1}{2}  (a_1+\frac{3}{4} )}^{\frac{3}{4}}  (x-\frac{3}{4} )^2 \, dx\Big)\\
&=\frac{1}{384} \Big(192 a_1^3-48 a_1^2-12 a_1+11\Big),
\end{align*}
the minimum value of which is $0.0208333$ and it occurs when $a_1=\frac 14$.

Thus, considering all the possible cases we see that the distortion error is smallest when $a_1=\frac 14$, and since $\frac 12(a_1+a_2)=\frac 34$, we have $a_2=\frac 54$. Thus, the optimal set of three-means is  $\set{\frac 14, \frac 34, \frac 54}$ with quantization error $V_3=\frac{1}{48}$ (also see Figure~\ref{Fig2}).
\end{proof}

\begin{prop}   \label{prop3}
The optimal set of four-means is  $\set{0.198223, 0.59467, 0.90533, 1.30178}$ with quantization error $V_4=0.01057$.
\end{prop}

\begin{proof}
Let $\ga:=\set{a_1, a_2, a_3, a_4}$ be an optimal set of four-means. Due to symmetry of the probability measure we can say that the elements in the optimal set will be symmetrically located on the line with respect to the element $\frac 34$, i.e.,  $0<a_1<a_2<\frac 34<a_3<a_4<\frac 32$, and $\frac 34$ is the midpoint of $a_2$ and $a_3$. The following cases can arise:

\tit{Case~1. $0<a_1<a_2\leq \frac 12$.}

In this case, due to symmetry the distortion error is given by
\begin{align*}
&\int \min_{a\in \ga}(x-a)^2 dP= 2 \Big(\frac{1}{2} \int_0^{\frac{1}{2} (a_1+a_2 )} (x-a_1 ){}^2 \, dx+\int_{\frac{1}{2}}^{\frac{3}{4}} (x-a_2 ){}^2 \, dx+\frac{1}{2} \int_{\frac{1}{2} (a_1+a_2 )}^{\frac{1}{2}} (x-a_2 ){}^2 \, dx\Big)\\
&=\frac{1}{96} \Big(24 a_1^3+24 a_2 a_1^2-24 a_2^2 a_1-24 a_2^3+96 a_2^2-84 a_2+23\Big),
\end{align*}
the minimum value of which is $0.0150463$ and it occurs when $a_1=0.166667$ and $a_2=0.5$.

\tit{Case~2. $0<a_1\leq \frac 12<a_2<\frac 34$.}

In this case, the following two subcases can occur.

\tit{Subcase~1. $\frac 12(a_1+a_2)\leq \frac 12.$}

In this subcase, due to symmetry the distortion error is given by
\begin{align*}
&\int \min_{a\in \ga}(x-a)^2 dP= 2 \Big(\frac{1}{2} \int_0^{\frac{1}{2}  (a_1+a_2 )}  (x-a_1 ){}^2 \, dx+\int_{\frac{1}{2}}^{\frac{3}{4}}  (x-a_2 ){}^2 \, dx+\frac{1}{2} \int_{\frac{1}{2}  (a_1+a_2 )}^{\frac{1}{2}}  (x-a_2 ){}^2 \, dx\Big)\\
&=\frac{1}{96} \Big(24 a_1^3+24 a_2 a_1^2-24 a_2^2 a_1-24 a_2^3+96 a_2^2-84 a_2+23\Big),
\end{align*}
the minimum value of which is $0.01057$ and it occurs when $a_1=0.198223$ and $a_2=0.59467$.

\tit{Subcase~2.  $\frac 12\leq \frac 12(a_1+a_2).$}

In this case, due to symmetry the distortion error is given by
\begin{align*}
&\int \min_{a\in \ga}(x-a)^2 dP= 2 \Big(\frac{1}{2} \int_0^{\frac{1}{2}}  (x-a_1 ){}^2 \, dx+\int_{\frac{1}{2}}^{\frac{1}{2}  (a_1+a_2 )}  (x-a_1 ){}^2 \, dx+\int_{\frac{1}{2}  (a_1+a_2 )}^{\frac{3}{4}}  (x-a_2 ){}^2 \, dx\Big)\\
&=\frac{1}{96}  (48 a_1^3+48  (a_2-1 ) a_1^2+ (24-48 a_2^2 ) a_1-48 a_2^3+144 a_2^2-108 a_2+23 ),
\end{align*}
the minimum value of which is $0.0169271$ and it occurs when $a_1=0.3125$ and $a_2=0.6875$.

\tit{Case~3. $\frac 12\leq a_1<a_2<\frac 34$.}

Notice that in this case we obtain
\begin{align*}
&\int \min_{a\in \ga}(x-a)^2 dP\geq \frac 2 2\int_0^{\frac{1}{2}} (x-\frac{1}{2})^2 \, dx=\frac{1}{24}=0.0416667,
\end{align*}
which is larger than the distortion errors obtained in at least one of the previous cases. So, this case cannot happen.

Thus, considering all the possible cases, we can deduce that the smallest distortion error is $V_4=0.01057$, and it occurs when $a_1=0.198223$ and $a_2=0.59467$. Since $\frac 12(a_2+a_3)=\frac 34$ and $\frac 12(a_1+a_4)=\frac 34$, we have $a_3=0.90533$ and $a_4=1.30178$. Thus, the optimal set of four-means is  $\set{0.198223, 0.59467, 0.90533, 1.30178}$ with quantization error $V_4=0.01057$, which is the proposition (also see Figure~\ref{Fig2}).
\end{proof}

\begin{prop}   \label{prop4}
The optimal set of five-means is  $\set{0.169821, 0.509464, \frac 34, 0.990536, 1.33018}$ with quantization error $V_5=0.00721728$.
\end{prop}

\begin{proof}
Let $\ga:=\set{a_1<a_2<a_3<a_4<a_5}$ be an optimal set of five-means.
As mentioned in Remark~\ref{rem111}, we can assume that $a_3=\frac 34$. The following cases can happen.

\tit{Case~1. $a_2\leq \frac 12$}

In this case the following subcases can happen. 

\tit{Subcase~1. $\frac 12(a_2+\frac 34)\leq \frac 12$.}

Due to symmetry the distortion error is given by
\begin{align*}
\int \min_{a\in \ga}(x-a)^2 dP&=  2 \Big(\frac{1}{2} \int_0^{\frac{1}{2} (a_1+a_2 )}  (x-a_1 ){}^2 \, dx+\frac{1}{2} \int_{\frac{1}{2}  (a_2+\frac{3}{4} )}^{\frac{1}{2}}  (x-\frac{3}{4} )^2 \, dx+\frac{1}{2} \int_{\frac{1}{2}  (a_1+a_2 )}^{\frac{1}{2}  (a_2+\frac{3}{4} )}  (x-a_2 ){}^2 \, dx\\
& \qquad \qquad \qquad  +\int_{\frac{1}{2}}^{\frac{3}{4}}  (x-\frac{3}{4} )^2 \, dx\Big)\\
&=\frac{1}{768} \Big(192 a_1^3+192 a_2 a_1^2-192 a_2^2 a_1+144 a_2^2-108 a_2+31\Big),
\end{align*}
the minimum value of which is $0.0162037$, and it occurs when $a_1=0.0833333$ and $a_2 =0.25$.

\tit{Subcase~2. $\frac 12 \leq \frac 12(a_2+\frac 34)$.}

Due to symmetry the distortion error is given by
\begin{align*}
\int \min_{a\in \ga}(x-a)^2 dP&=  2 \Big(\frac{1}{2} \int_0^{\frac{1}{2}  (a_1+a_2 )}  (x-a_1 ){}^2 \, dx+\int_{\frac{1}{2}  (a_2+\frac{3}{4} )}^{\frac{3}{4}}  (x-\frac{3}{4} )^2 \, dx+\int_{\frac{1}{2}}^{\frac{1}{2}  (a_2+\frac{3}{4} )}  (x-a_2 ){}^2 \, dx\\
&\qquad \qquad \qquad +\frac{1}{2} \int_{\frac{1}{2}  (a_1+a_2 )}^{\frac{1}{2}}  (x-a_2 ){}^2 \, dx\Big)\\
&=\frac{1}{384} \Big(96 a_1^3+96 a_2 a_1^2-96 a_2^2 a_1+96 a_2^3-48 a_2^2-12 a_2+11\Big),
\end{align*}
the minimum value of which is $0.0162037$, and it occurs when $a_1=0.0833333$ and $a_2 =0.25$.

\tit{Case~2. $0<a_1\leq \frac 12<a_2<\frac 34$.}

In this case the following subcases can happen.

\tit{Subcase~1. $\frac 12(a_1+a_2)\leq \frac 12$.}

Due to symmetry the distortion error is given by
\begin{align*}
\int \min_{a\in \ga}(x-a)^2 dP&=2 \Big(\int_{\frac{1}{2}  (a_2+\frac{3}{4} )}^{\frac{3}{4}}  (x-\frac{3}{4} )^2 \, dx+\frac{1}{2} \int_0^{\frac{1}{2} (a_1+a_2 )} (x-a_1 ){}^2 \, dx+\int_{\frac{1}{2}}^{\frac{1}{2}  (a_2+\frac{3}{4} )}  (x-a_2 ){}^2 \, dx\\
&\qquad \qquad \qquad +\frac{1}{2} \int_{\frac{1}{2}  (a_1+a_2 )}^{\frac{1}{2}} (x-a_2 ){}^2 \, dx\Big)\\
&=\frac{1}{384} \Big(96 a_1^3+96 a_2 a_1^2-96 a_2^2 a_1+96 a_2^3-48 a_2^2-12 a_2+11\Big),
\end{align*}
the minimum value of which is $0.00721728$, and it occurs when $a_1=0.169821$ and $a_2 =0.509464$.

\tit{Subcase~2. $\frac 12 \leq \frac 12(a_1+a_2)$.}

Due to symmetry the distortion error is given by
\begin{align*}
\int \min_{a\in \ga}(x-a)^2 dP&=  2 \Big(\int_{\frac{1}{2}  (a_2+\frac{3}{4} )}^{\frac{3}{4}}  (x-\frac{3}{4} )^2 \, dx+\frac{1}{2} \int_0^{\frac{1}{2}}  (x-a_1 ){}^2 \, dx+\int_{\frac{1}{2}}^{\frac{1}{2}  (a_1+a_2 )}  (x-a_1 ){}^2 \, dx\\
&\qquad \qquad \qquad +\int_{\frac{1}{2}  (a_1+a_2 )}^{\frac{1}{2}  (a_2+\frac{3}{4} )}  (x-a_2 ){}^2 \, dx\Big)\\
&=\frac{1}{384} \Big(192 a_1^3+192  (a_2-1 ) a_1^2-96  (2 a_2^2-1 ) a_1+144 a_2^2-108 a_2+11\Big),
\end{align*}
the minimum value of which is $0.0167955$, and it occurs when $a_1=0.315741$ and $a_2 =0.684259$.

\tit{Case~3. $\frac 12<a_1<a_2<\frac 34$.}

Due to symmetry the distortion error is given by
\begin{align*}
\int \min_{a\in \ga}(x-a)^2 dP>2\times \frac 12 \int_0^{\frac{1}{2}} (x-\frac{1}{2})^2 \, dx=0.0416667,
\end{align*}
which is larger than the distortion error that arises in at least one of the previous cases.

Taking into consideration all the above possible cases, we see that the quantization error for optimal set of five-means is $V_5=0.00721728$, and it occurs when $a_1=0.169821$ and $a_2 =0.509464$. Due to symmetry, we have $a_4=0.990536$ and $a_5=1.33018$. Thus, the proof of the proposition is complete (also see Figure~\ref{Fig2}).
\end{proof}

Let us now prove the following lemma.

\begin{lemma} \label{lemma3}
Let $\ga_n$ be an optimal set of $n$-means for $n\geq 4$. Then, $\ga_n$ contains elements from both the open intervals $(0, \frac 12)$ and $(\frac 12, \frac 34)$.
\end{lemma}

\begin{proof}
By Propositions~\ref{prop3} and Proposition~\ref{prop4}, the lemma is true for $n=4$ and $n=5$. Let us now prove the lemma for $n\geq 6$. We prove it by contradiction. Recall Remark~\ref{rem111}, and also recall that for $n\geq 6$, we have $V_n\leq V_6<V_5$. For $n\geq 6$, if $\ga_n$ does not contain any element from the open interval $(\frac 12, \frac 34)$, then due to symmetry we have
\[V_n\geq 2 \int_{\frac{1}{2}}^{\frac{3}{4}} (x-\frac{1}{2})^2 \, dx=\frac{1}{96}=0.0104167>V_5,\]
which leads to a contradiction. For $n\geq 6$, if $\ga_n$ does not contain any element from the open interval $(0, \frac 12)$, then due to symmetry we have
\[V_n\geq \frac{2}{2} \int_0^{\frac{1}{2}} (x-\frac{1}{2})^2 \, dx =\frac{1}{24}=0.0416667>V_5,\]
which is a contradiction.
Hence, we can conclude that the lemma is also true for $n\geq 6$. Thus, the proof of the lemma is complete.
\end{proof}

\begin{prop}   \label{prop6}
The optimal set of six-means is  $\set{0.125, 0.375, 0.625, 0.875, 1.125, 1.375}$ with quantization error $V_6=0.00520833$.
\end{prop}

\begin{proof}
Let $\ga:=\set{a_1, a_2, a_3, a_4, a_5, a_6}$ be an optimal set of six-means. Due to symmetry of the probability measure we can say that the elements in the optimal set will be symmetrically located on the line with respect to the element $\frac 34$, i.e.,  $0<a_1<a_2<a_3<\frac 34<a_4<a_5<a_6<\frac 32$, and $\frac 34$ is the midpoint of $a_3$ and $a_4$.
By Lemma~\ref{lemma3}, we can say that $a_1<\frac 12$, and $\frac 12<a_3$.

The following cases can arise:

\tit{Case~1. $0<a_1<a_2\leq \frac 12$.}

The following two subcases can occur.

\tit{Subcase~1. $0<a_1<a_2<\frac 12(a_2+a_3)\leq \frac 12<a_3<\frac 34$.}

In this subcase, due to symmetry the distortion error is given by
\begin{align*}
&\int \min_{a\in \ga}(x-a)^2 dP= 2 \Big(\frac{1}{2} \int_0^{\frac{1}{2}  (a_1+a_2 )}  (x-a_1 ){}^2 \, dx+\frac{1}{2} \int_{\frac{1}{2}  (a_1+a_2 )}^{\frac{1}{2} (a_2+a_3 )}  (x-a_2 ){}^2 \, dx+\int_{\frac{1}{2}}^{\frac{3}{4}}  (x-a_3 ){}^2 \, dx\\
&\qquad \qquad \qquad +\frac{1}{2} \int_{\frac{1}{2}  (a_2+a_3 )}^{\frac{1}{2}} (x-a_3 ){}^2 \, dx\Big)\\
&=\frac{1}{96} \Big(24 a_1^3+24 a_2 a_1^2-24 a_2^2 a_1-24 a_3^3-24  (a_2-4 ) a_3^2+12  (2 a_2^2-7 ) a_3+23\Big),
\end{align*}
the minimum value of which is $0.00520833$ and it occurs when $a_1=0.125$, $a_2=0.375$ and $a_3=0.625$.

\tit{Subcase~2. $0<a_1<a_2<\frac 12\leq \frac 12(a_2+a_3)<a_3<\frac 34$.}

In this subcase, due to symmetry the distortion error is given by
\begin{align*}
&\int \min_{a\in \ga}(x-a)^2 dP= 2 \Big(\frac{1}{2} \int_0^{\frac{1}{2}  (a_1+a_2 )}  (x-a_1 ){}^2 \, dx+\int_{\frac{1}{2}}^{\frac{1}{2}  (a_2+a_3 )} (x-a_2 ){}^2 \, dx+\frac{1}{2} \int_{\frac{1}{2}  (a_1+a_2 )}^{\frac{1}{2}}  (x-a_2 ){}^2 \, dx\\
&\qquad \qquad \qquad +\int_{\frac{1}{2}  (a_2+a_3 )}^{\frac{3}{4}}  (x-a_3 ){}^2 \, dx\Big)\\
&=\frac{1}{96} \Big(24 a_1^3+24 a_2 a_1^2-24 a_2^2 a_1+24 a_2^3-48 a_3^3+144 a_3^2+48 a_2^2  (a_3-1 )-108 a_3\\
&=\frac{1}{96} \Big(24 a_1^3+24 a_2 a_1^2-24 a_2^2 a_1+24 a_2^3-48 a_3^3+144 a_3^2\\
&\hspace{ 3 in }+a_2  (24-48 a_3^2 )+23\Big),
\end{align*}
the minimum value of which is $0.00520833$ and it occurs when $a_1=0.125$, $a_2=0.375$ and $a_3=0.625$.

\tit{Case~2. $0<a_1< \frac 12< a_2<a_3<\frac 34$.}

In this case, proceeding as before considering the two subcases: $\frac{1}{2} \left(a_1+a_2\right)\leq \frac{1}{2}<a_2<a_3<\frac{3}{4}$, and $\frac{1}{2}\leq \frac{1}{2} \left(a_1+a_2\right)<a_2<a_3<\frac{3}{4}$, it can be shown that the distortion error is larger than the distortion error obtained in Case~1. Therefore, this case cannot happen.

Hence, the quantization error for six-means is $V_6=0.00520833$, and it occurs when $a_1=0.125, \, a_2=0.375, \, a_3=0.625$. Due to symmetry, we have  $a_4=0.875, \, a_5=1.125$, and $a_6=1.375$. Thus, the proof of the proposition is complete (also see Figure~\ref{Fig2}).
 \end{proof}

\begin{figure}
\centerline{\includegraphics[width=10 in, height=5.2 in]{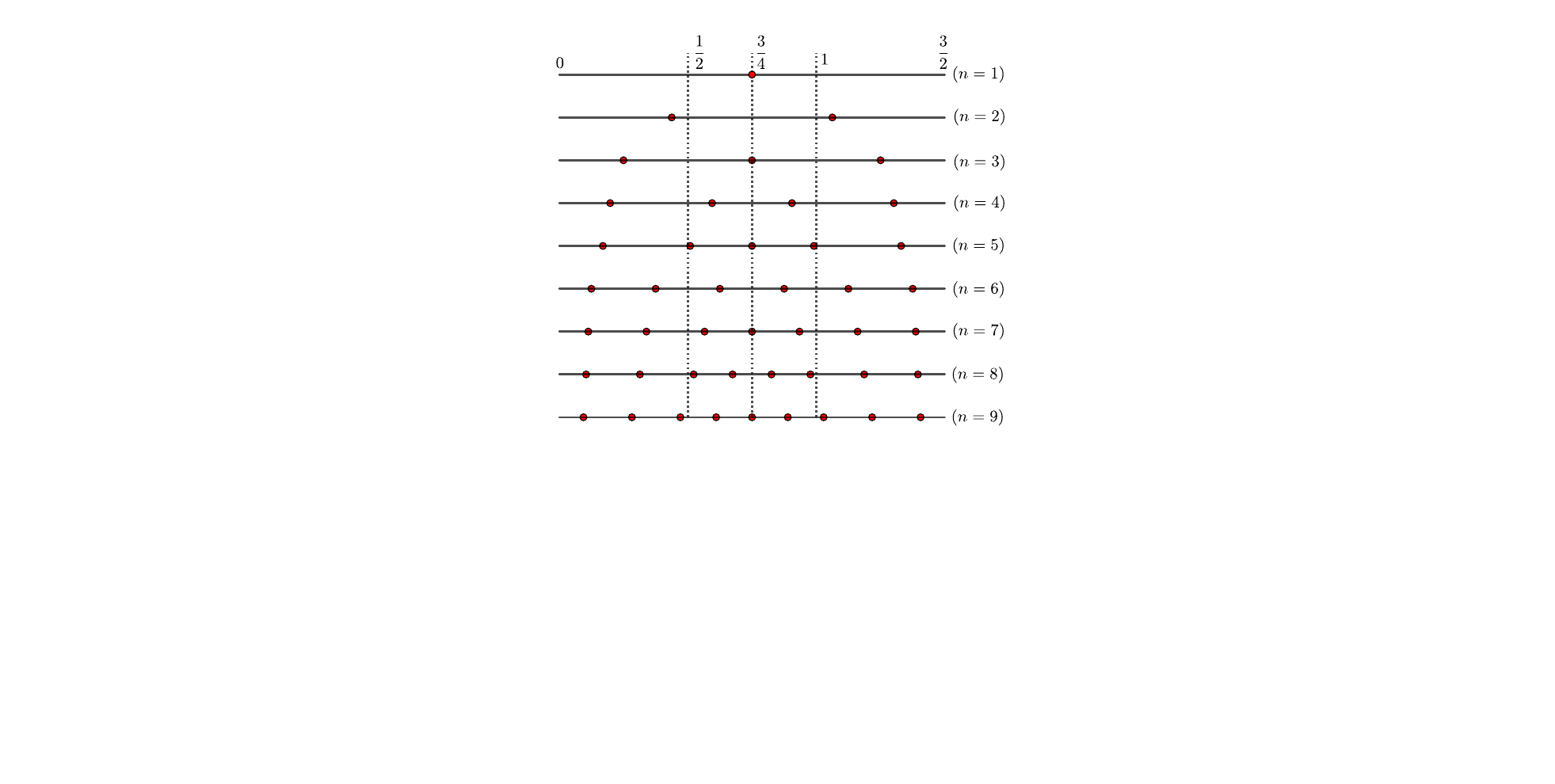}}
\vspace{-2.5 in}
\caption{Elements in the optimal sets of $n$-means for $1\leq n\leq 9$.} \label{Fig2}
\end{figure}

In the following section we calculate the optimal sets of $n$-means and the $n$th quantization errors for all $n\geq 5$.

\section{Optimal sets of $n$-means and the $n$th quantization errors for all $n\geq 5$} \label{sec3} 
Let $n\geq 5$ be a positive integer. By Remark~\ref{rem111}, we know that if $n$ is odd, an optimal set $\ga_n$ of $n$-means always contains the element $\frac 34$. Notice that whether $n$ is an even or an odd positive integer, it is enough to find the elements in an optimal set which are to  the left side of $\frac 34$, i.e., which are belonged to the interval $(0, \frac 34)$; the remaining elements in $\ga_n$ can be obtained by taking the reflections with respect to the element $\frac 34$.
By Lemma~\ref{lemma3}, an optimal set $\ga_n$ contains elements from both the open intervals $(0, \frac 12)$ and $(\frac 12, \frac 34)$. Thus, there exist two positive integers $k:=k(n)$ and $m:=m(n)$ such that
\begin{align*}
\ga_n\ii [0, \frac 12]&:=\set{a_1<a_2<\cdots<a_k}, \te{ and }\\
\ga_n\ii (\frac 12, \frac 34]&:=\left\{\begin{array}{cc}
 \set{b_1<b_2<\cdots<b_m} &\te{ if $n$ is even,}\\
  \set{b_1<b_2<\cdots<b_m<\frac 34} &\te{ if $n$ is odd.}
  \end{array}
  \right.
  \end{align*}
Observe that in the above expression, if $n$ is even, then $2(k+m)=n$; and if $n$ is odd, then $2(k+m)+1=n$. Notice that the following two cases can happen: either $\frac 12(a_k+b_1)\leq \frac 12$, or $\frac 12\leq \frac 12(a_k+b_1)$.
Whether $n$ is even or odd, let $V1(k, m)$ be the $n$th quantization error when  $\frac{a_k+b_1}{2}\leq \frac 12$, and $V2(k, m)$ be the $n$th quantization error when  $\frac 12\leq \frac{a_k+b_1}{2}$. The optimal sets of $n$-means and the $n$th quantization errors for $n=1, 2, 3, 4, 5$ are given in the previous sections. The following propositions will give the optimal sets of $n$-means and the $n$th quantization errors for all $n\geq 6$.
\begin{prop} \label{prop71}
Let $k\geq 2$ and $m=1$. Then, if $\frac 12(a_k+b_1)\leq \frac 12$, we have $a_j=\frac{(2j-1)(a_k+b_1)}{4k}$ for $1\leq j\leq k$, and
\begin{align*}
b_1=\left\{\begin{array}{cc}
E(X : X\in [\frac 12(a_k+b_1), \frac 34]) & \te{ if $n$ is even}, \\
E(X : X\in [\frac 12(a_k+b_1), \frac 12(b_1+\frac 34)]) & \te{ if $n$ is odd};
\end{array}
\right.
\end{align*}
and if $\frac 12\leq \frac 12(a_k+b_1)$, we have $a_j=\frac{(2j-1)(a_{k-1}+a_k)}{4 (k-1)}$ for $1\leq j\leq (k-1)$, $a_k=E(X : X\in [\frac{1}{2} (a_{k-1}+a_k), \frac 12(a_k+b_1)])$, and
\begin{align*}
b_1=\left\{\begin{array}{cc}
E(X : X\in [\frac 12(a_k+b_1), \frac 34]) & \te{ if $n$ is even}, \\
E(X : X\in [\frac 12(a_k+b_1), \frac 12(b_1+\frac 34)]) & \te{ if $n$ is odd}.
\end{array}
\right.
\end{align*}
The quantization errors for $n$-means are given by
\begin{align*}
 V1(k, 1)=\left\{\begin{array}{cc}
 2\Big(\frac{ (a_k+b_1 ){}^3}{192 k^2}+\frac{1}{2} \int_{\frac{1}{2}  (a_k+b_1 )}^{\frac{1}{2}}  (x-b_1 ){}^2 \, dx+\int_{\frac{1}{2}}^{\frac{3}{4}}  (x-b_1 ){}^2 \, dx\Big) & \te{ if $n$ is even},\\
 2\Big(\frac{ (a_k+b_1 ){}^3}{192 k^2}+\frac{1}{2} \int_{\frac{1}{2}  (a_k+b_1 )}^{\frac{1}{2}}  (x-b_1 ){}^2 \, dx+\int_{\frac{1}{2}}^{\frac{1}{2} (b_1+\frac{3}{4} )}  (x-b_1 ){}^2 \, dx\\
 +\int_{\frac{1}{2}  (b_1+\frac{3}{4} )}^{\frac{3}{4}}  (x-\frac{3}{4} )^2 \, dx\Big) & \te{ if $n$ is odd};
 \end{array}
 \right.
 \end{align*}
 and
 \begin{align*}
 V2(k, 1)=\left\{\begin{array}{cc}
 2\Big(\frac{(a_{k-1}+a_k){}^3}{192 (k-1)^2}+\frac{1}{2} \int_{\frac{1}{2} (a_{k-1}+a_k)}^{\frac{1}{2}} (x-a_k){}^2 \, dx+\int_{\frac{1}{2}}^{\frac{1}{2} (a_k+b_1)} (x-a_k){}^2 \, dx\\
 +\int_{\frac{1}{2} (a_k+b_1)}^{\frac{3}{4}} (x-b_1){}^2 \, dx\Big) & \te{ if $n$ is even},\\
 2\Big(\frac{ (a_{k-1}+a_k ){}^3}{192 (k-1)^2}+\frac{1}{2} \int_{\frac{1}{2}  (a_{k-1}+a_k )}^{\frac{1}{2}}  (x-a_k ){}^2 \, dx+\int_{\frac{1}{2}}^{\frac{1}{2}  (a_k+b_1 )} (x-a_k ){}^2 \, dx\\
 +\int_{\frac{1}{2}  (a_k+b_1 )}^{\frac{1}{2}  (b_1+\frac{3}{4} )}  (x-b_1 ){}^2 \, dx+\int_{\frac{1}{2}  (b_1+\frac{3}{4} )}^{\frac{3}{4}}  (x-\frac{3}{4} )^2 \, dx\Big) & \te{ if $n$ is odd}.
 \end{array}
 \right.
 \end{align*}
 \end{prop}
\begin{proof}
If $\frac{a_k+b_1}{2}\leq \frac 12$, then $a_1, a_2, \cdots, a_{k}$ are uniformly distributed over the closed interval $[0, \frac{a_k+b_1}{2}]$; on the other hand,
if $\frac 12 \leq \frac{a_k+b_1}{2}$, then $a_1, a_2, \cdots, a_{k-1}$ are uniformly distributed over the closed interval $[0, \frac{a_{k-1}+a_k}2]$.
 Thus, by Proposition~\ref{prop0} and Proposition~\ref{prop55}, the expressions for $a_j$ and $b_j$ can be obtained. With the help of the formula given in Proposition~\ref{prop55}, the quantization errors are also obtained as routine.
\end{proof}
\begin{prop} \label{prop72}
Let $k=1$ and $m\geq 2$. Then, if $\frac 12(a_1+b_1)\leq \frac 12$, we have $a_1=E(X : X\in [0, \frac 12(a_1+b_1)]$, $b_1=E(X : X\in [\frac 12(a_1+b_1), \frac 12(b_1+b_2)])$, and
\begin{align*}
b_{1+j}=\left\{\begin{array}{cc}
\frac 12(b_1+b_2)+\frac{(2j-1)}{2(m-1)}(\frac 34-\frac 12(b_1+b_2)) \te{ for } 1\leq j\leq m-1 & \te{ if $n$ is even}, \\
\frac 12(b_1+b_2)+\frac{(2j-1)}{2(m-1)}(\frac 34-b_2) \te{ for } 1\leq j\leq m-1 & \te{ if $n$ is odd};
\end{array}
\right.
\end{align*}
and if $\frac 12\leq \frac 12(a_1+b_1)$, we have $a_1=E(X : X\in [0, \frac 12(a_1+b_1)]$, and
\begin{align*}
b_{j}=\left\{\begin{array}{cc}
\frac 12(a_1+b_1)+\frac{(2j-1)}{2m}(\frac 34-\frac 12(a_1+b_1))$ for $1\leq j\leq m & \te{ if $n$ is even}, \\
\frac 12(a_1+b_1)+\frac{(2j-1)}{2m}(\frac 34-b_1) \te{ for } 1\leq j\leq m  & \te{ if $n$ is odd};
\end{array}
\right.
\end{align*}
The quantization errors for $n$-means are given by
\begin{align*}
 V1(1, m)=\left\{\begin{array}{cc}
  2\Big(\frac{1}{2} \int_0^{\frac{1}{2}  (a_1+b_1 )}  (x-a_1 ){}^2 \, dx+\frac{1}{2} \int_{\frac{1}{2}  (a_1+b_1 )}^{\frac{1}{2}}  (x-b_1 ){}^2 \, dx\\
  +\int_{\frac{1}{2}}^{\frac{1}{2}  (b_1+b_2 )}  (x-b_1 ){}^2 \, dx+\frac{(3-2(b_1+b_2))^3}{768(m-1)^2}\Big) & \te{ if $n$ is even},\\
 2\Big(\frac{1}{2} \int_0^{\frac{1}{2}  (a_1+b_1 )}  (x-a_1 ){}^2 \, dx+\frac{1}{2} \int_{\frac{1}{2}  (a_1+b_1 )}^{\frac{1}{2}}  (x-b_1 ){}^2 \, dx\\+\int_{\frac{1}{2}}^{\frac{1}{2}  (b_1+b_2 )}  (x-b_1 ){}^2 \, dx+\frac{ (3-4 b_2 ){}^3}{768 (m-1)^2}+\frac{1}{24}  (b_2-b_1 ){}^3\Big) & \te{ if $n$ is odd};
 \end{array}
 \right.
 \end{align*}
 and
 \begin{align*}
 V2(1, m)=\left\{\begin{array}{cc}
 2\Big(\frac{1}{2} \int_0^{\frac{1}{2}} \left(x-a_1\right){}^2 \, dx+\int_{\frac{1}{2}}^{\frac{1}{2} \left(a_1+b_1\right)} \left(x-a_1\right){}^2 \, dx
 +\frac{\left(3-2 (a_1+b_1)\right){}^3}{768 m^2}\Big) & \te{ if $n$ is even},\\
 2\Big(\frac{1}{2} \int_0^{\frac{1}{2}} \left(x-a_1\right){}^2 \, dx+\int_{\frac{1}{2}}^{\frac{1}{2} \left(a_1+b_1\right)} \left(x-a_1\right){}^2 \, dx\\
 +\frac{(3-4b_1){}^3}{768 m^2}+\frac{1}{24} (b_1-a_1){}^3\Big) & \te{ if $n$ is odd}.
 \end{array}
 \right.
 \end{align*}
\end{prop}

\begin{proof}
If $\frac 12(a_1+b_1)\leq \frac 12$ and $n$ is even, then $b_2, b_3, \cdots, b_m$ are uniformly distributed over the closed interval $[\frac 12(b_1+b_2), \frac 34]$, and so by Proposition~\ref{prop55}, the expressions for $b_{i+j}$, where $1\leq j\leq (m-1)$, and the corresponding quantization error can be obtained. On the other hand, if $\frac 12(a_1+b_1)\leq \frac 12$ and $n$ is odd, then as $\set{b_2<b_3<\cdots<b_m<\frac 34}$ is the set of optimal quantizers with respect to the probability distribution $P(\cdot|_{[\frac 12(b_1+b_2), \frac 34]})$ with constant density $f(x)=1$ for all $x\in [\frac 12(b_1+b_2), \frac 34]$, the expressions for $b_j$, and the corresponding quantization error can be obtained using Proposition~\ref{prop56}.
Likewise, if $\frac 12\leq \frac 12(a_1+b_1)$, using Proposition~\ref{prop55} and Proposition~\ref{prop56},  we get the expressions for the optimal quantizers and the corresponding quantization error.
\end{proof}

\begin{prop} \label{prop73}
Let $k\geq 2$ and $m\geq 2$. Then, if $\frac 12(a_k+b_1)\leq \frac 12$, we have $a_j=\frac{(2j-1)(a_k+b_1)}{4k}$ for $1\leq j\leq k$,
$b_1=E(X : X\in [\frac 12(a_k+b_1), \frac 12(b_1+b_2)])$, and
\begin{align*}
b_{1+j}=\left\{\begin{array}{cc}
\frac 12(b_1+b_2)+\frac{(2j-1)}{2(m-1)}(\frac 34-\frac 12(b_1+b_2)) \te{ for } 1\leq j\leq m-1 & \te{ if $n$ is even}, \\
\frac 12(b_1+b_2)+\frac{(2j-1)}{2(m-1)}(\frac 34-b_2) \te{ for } 1\leq j\leq m-1 & \te{ if $n$ is odd};
\end{array}
\right.
\end{align*}
and if $\frac 12\leq \frac 12(a_k+b_1)$, we have $a_j=\frac{(2j-1)(a_{k-1}+a_k)}{4 (k-1)}$ for $1\leq j\leq (k-1)$, $a_k=E(X : X\in [\frac{1}{2} (a_{k-1}+a_k), \frac 12(a_k+b_1)])$, and
\begin{align*}
b_{j}=\left\{\begin{array}{cc}
\frac 12(a_k+b_1)+\frac{(2j-1)}{2m}(\frac 34-\frac 12(a_k+b_1))$ for $1\leq j\leq m & \te{ if $n$ is even}, \\
\frac 12(a_k+b_1)+\frac{(2j-1)}{2m}(\frac 34-b_1) \te{ for } 1\leq j\leq m& \te{ if $n$ is odd};
\end{array}
\right.
\end{align*}
The quantization errors for $n$-means are given by
\begin{align*}
 V1(k, m)=\left\{\begin{array}{cc}
 2\Big(\frac{ (a_k+b_1 ){}^3}{192 k^2}+\frac{1}{2} \int_{\frac{1}{2}  (a_k+b_1 )}^{\frac{1}{2}}  (x-b_1 ){}^2 \, dx+\int_{\frac{1}{2}}^{\frac{1}{2} (b_1+b_2)}  (x-b_1 ){}^2 \, dx\\
 +\frac{(3-2(b_1+b_2))^3}{768(m-1)^2}\Big) & \te{ if $n$ is even},\\
 2\Big(\frac{ (a_k+b_1 ){}^3}{192 k^2}+\frac{1}{2} \int_{\frac{1}{2}  (a_k+b_1 )}^{\frac{1}{2}}  (x-b_1 ){}^2 \, dx\\+\int_{\frac{1}{2}}^{\frac{1}{2}  (b_1+b_2 )}  (x-b_1 ){}^2 \, dx+\frac{ (3-4 b_2 ){}^3}{768 (m-1)^2}+\frac{1}{24}  (b_2-b_1 ){}^3\Big) & \te{ if $n$ is odd};
 \end{array}
 \right.
 \end{align*}
 and
 \begin{align*}
 V2(k, m)=\left\{\begin{array}{cc}
 2\Big(\frac{(a_{k-1}+a_k){}^3}{192 (k-1)^2}+\frac{1}{2} \int_{\frac{1}{2} (a_{k-1}+a_k)}^{\frac{1}{2}} (x-a_k){}^2 \, dx+\int_{\frac{1}{2}}^{\frac{1}{2} (a_k+b_1)} (x-a_k){}^2 \, dx\\
 +\frac{\left(3-2 (a_k+b_1)\right){}^3}{768 m^2}\Big) & \te{ if $n$ is even},\\
 2\Big(\frac{ (a_{k-1}+a_k ){}^3}{192 (k-1)^2}+\frac{1}{2} \int_{\frac{1}{2}  (a_{k-1}+a_k )}^{\frac{1}{2}}  (x-a_k ){}^2 \, dx+\int_{\frac{1}{2}}^{\frac{1}{2}  (a_k+b_1 )} (x-a_k ){}^2 \, dx\\
 +\frac{(3-4b_1){}^3}{768 m^2}+\frac{1}{24} (b_1-a_k){}^3\Big) & \te{ if $n$ is odd}.
 \end{array}
 \right.
 \end{align*}
 \end{prop}

\begin{proof}
Notice that Proposition~\ref{prop73} is a mixture of Proposition~\ref{prop71} and Proposition~\ref{prop72}, and thus the proof follows in the similar lines.
\end{proof}

\begin{lemma} \label{lemma24}
For any even positive integer $n\geq 4$, let $\ga_n$ be an optimal set of $n$-means for $P$. Assume that $\te{card}(\ga_n\ii [0, \frac 12])=k:=k(n)$ and $\te{card}(\ga_n\ii (\frac 12, \frac 34))=m:=m(n)$ for some positive integers $k$ and $m$. Then, either $\te{card}(\ga_{n+2}\ii [0, \frac 12])=k+1$ and $\te{card}(\ga_{n+2}\ii (\frac 12, \frac 34))=m$, or $\te{card}(\ga_{n+2}\ii [0, \frac 12])=k$ and $\te{card}(\ga_{n+2}\ii (\frac 12, \frac 34))=m+1$.
\end{lemma}

\begin{proof}
For any even positive integer $n\geq 4$, let $\te{card}(\ga_n\ii [0, \frac 12])=k:=k(n)$ and $\te{card}(\ga_n\ii (\frac 12, \frac 34))=m:=m(n)$ for some positive integers $k$ and $m$. Let $V(k(n), m(n))$ be the corresponding distortion error.
By Proposition~\ref{prop3} and Proposition~\ref{prop6}, we know that $\te{card}(\ga_4\ii [0, \frac 12])=1$, $\te{card}(\ga_4\ii(\frac 12, \frac 34))=1$, $\te{card}(\ga_6\ii [0, \frac 12])=2$, and $\te{card}(\ga_6\ii(\frac 12, \frac 34))=1$
Thus, the lemma is true for $n=4$.
Let the lemma be true for $n=N$ for some even positive integer $N\geq 4$. Then, $\te{card}(\ga_N\ii [0, \frac 12])=k(N) \te{ and } \te{card}(\ga_N\ii (\frac 12, \frac 34))=m(N)$ imply that
either $\te{card}(\ga_{N+2}\ii [0, \frac 12])=k(N)+1$ and $\te{card}(\ga_{N+2}\ii (\frac 12, \frac 34))=m(N)$, or $\te{card}(\ga_{N+2}\ii [0, \frac 12])=k(N)$ and $\te{card}(\ga_{N+2}\ii (\frac 12, \frac 34))=m(N)+1$. Suppose that
$\te{card}(\ga_{N+2}\ii [0, \frac 12])=k(N)+1$ and $\te{card}(\ga_{N+2}\ii (\frac 12, \frac 34))=m(N)$ hold. Now, for the given $N$, by calculating the distortion errors $V(\ell, N+4-\ell)$ for all $1\leq \ell \leq N+3$, we see that the distortion error is smallest if $\te{card}(\ga_{N+4}\ii [0, \frac 12])=k(N)+2 \te{ and } \te{card}(\ga_{N+2}\ii (\frac 12, \frac 34)) =m(N), \te{ or if } \te{card}(\ga_{N+4}\ii [0, \frac 12])=k(N)+1 \te{ and } \te{card}(\ga_{N+4}\ii (\frac 12, \frac 34))=m(N)+1$, i.e., the lemma is true for $n=N+2$ whenever it is true for $n=N$. Similarly, we can show that the lemma is true for $n=N+2$ if $\te{card}(\ga_{N+2}\ii [0, \frac 12])=k(N) \te{ and } \te{card}(\ga_{N+2}\ii (\frac 12, \frac 34))=m(N)+1$ hold.  Thus, by the induction principle, the proof of the lemma is complete.
 \end{proof}

Proceeding in the similar lines as Lemma~\ref{lemma24}, the following lemma can be proved.

\begin{lemma} \label{lemma25}
For any odd positive integer $n\geq 5$, let $\ga_n$ be an optimal set of $n$-means for $P$. Assume that $\te{card}(\ga_n\ii [0, \frac 12])=k:=k(n)$ and $\te{card}(\ga_n\ii (\frac 12, \frac 34))=m:=m(n)$ for some positive integers $k$ and $m$. Then, either $\te{card}(\ga_{n+2}\ii [0, \frac 12])=k+1$ and $\te{card}(\ga_{n+2}\ii (\frac 12, \frac 34))=m$, or $\te{card}(\ga_{n+2}\ii [0, \frac 12])=k$ and $\te{card}(\ga_{n+2}\ii (\frac 12, \frac 34))=m+1$.
\end{lemma}

 \begin{defi} \label{defi000}
 Define a real valued function $F(k, m)$ on the domain $\D N\times \D N \setminus\set{(1,1)}$ such that
 \[F(k, m)=\min\set{V1(k,m), V2(k, m)},\]
 where $V1(k, m)$ and $V2(k, m)$ are the distortion errors as defined before.
 \end{defi}

\begin{defi} \label{difi21}
Define the sequence $\set{a(n)}$ such that
\[a(n):=\left\{\begin{array}{cc}
0 & \te{ if } n=1,\\
\Big\lfloor \frac{\lfloor \frac{n}{2}\rfloor }{\sum _{k=1}^{\lfloor \frac{n}{2}\rfloor } \frac{1}{k^2}}\Big\rfloor & \te{ if } n\geq 2.
\end{array}
\right.\]

i.e.,
\begin{align*}
 \set{a(n)}_{n=1}^\infty=&\set{0, 1,1,1,1,2,2,2,2,3,3,4,4,4,4,5,5,5,5,6,6,7,7,7,7,8,8,8,8,9,9,\\
 & 10,10,10,10,11,11,11,11,12,12,13,13,13,13,14,14,14,14,15  \cdots},
\end{align*}
where $\lfloor x\rfloor$ represents the greatest integer not exceeding $x$.
\end{defi}
\begin{remark} \label{rem90}
Let $n\geq 5$. Let $k$ and $m:=\lfloor\frac n2\rfloor-k$ be the positive integers such that $\te{card}(\ga_n\ii [0,\frac 12])=k$ and $\te{card}(\ga_n\ii(\frac 12, \frac 34))=m$. Thus, for a given $n\geq 5$, if we know the exact values of $k$ and $m$, then by Proposition~\ref{prop71} through Proposition~\ref{prop73}, we can easily determine the optimal sets $\ga_n$ and the corresponding quantization error.
\end{remark}
The following algorithm helps us to calculate the exact value of $k$ and so, $m=\lfloor\frac n2\rfloor-k$.
\subsection{\tbf{Algorithm.}} \label{Algo} Let $n\geq 5$ and $F(k, m)$ be the function defined by Definition~\ref{defi000}, and let $\set{a(n)}$ be the sequence  defined by Definition~\ref{difi21}. Then, the algorithm runs as follows:

$(i)$ Write $k:=a(n)$ and calculate $F(k, \lfloor\frac n2\rfloor-k)$.

$(ii)$ If $F(k-1, \lfloor\frac n2\rfloor-k+1)<F(k, \lfloor\frac n2\rfloor-k)$ replace $k$ by $k-1$ and return, else step $(iii)$.

$(iii)$ If $F(k+1,\lfloor\frac n2\rfloor-k-1)<F(k, \lfloor\frac n2\rfloor-k)$ replace $k$ by $k+1$ and return, else step $(iv)$.

$(iv)$ End.

When the algorithm ends, then the value of $k$, obtained, is the exact value of $k$ that $\ga_n$ contains from the closed interval $[0, \frac 12]$.

\subsubsection*{\tbf{Optimal sets of $n$-means and the $n$th quantization errors for all positive integers $n\geq 5$.}}
If \( n = 5 \), then \( a(n) = 1 \), and the algorithm yields \( k = 1 \), indicating that an optimal set \( \gamma_5 \) of five-means contains one element from the closed interval \( [0, \tfrac{1}{2}] \), which is consistent with Proposition~\ref{prop4}. For \( n = 6 \), we have \( a(n) = 2 \) and the algorithm also gives \( k = 2 \), in agreement with Proposition~\ref{prop6}. Similarly, for \( n = 9 \), \( a(n) = 2 \) and the algorithm returns \( k = 3 \); and for \( n = 50 \), \( a(n) = 15 \) with the algorithm also producing \( k = 15 \). In the case of \( n = 1001 \), we find \( a(n) = 304 \), and the algorithm yields \( k = 307 \). These examples demonstrate that the combination of the sequence and the algorithm allows for the straightforward determination of the exact values of \( k \) and \( m := \lfloor \frac{n}{2} \rfloor - k \) for any integer \( n \geq 5 \). Consequently, as noted in Remark~\ref{rem90}, one can systematically obtain the optimal sets of \( n \)-means and the associated quantization errors for all integers \( n \geq 5 \) (see also Figure~\ref{Fig2}).

 \section{Conclusion and Future Work} \label{sec5}

In this paper, we investigated optimal quantization for a class of mixed probability distributions formed from two uniform distributions with partially overlapping supports. Specifically, we analyzed a symmetric mixture of uniform distributions on $[0,1]$ and $\left[\frac{1}{2}, \frac{3}{2}\right]$ with equal weights. We determined the optimal sets of $n$-means and corresponding quantization errors for $1 \leq n \leq 6$ explicitly and established a comprehensive framework to compute these quantities for all $n \geq 5$. Key results include several structural lemmas, a classification of quantizer configurations based on the location of their Voronoi boundaries, and a deterministic algorithm to compute optimal sets for general $n$.

The findings enhance our understanding of quantization behavior in the presence of overlapping supports, offering both theoretical insights and practical computational methods. This work also demonstrates how symmetry and partitioning properties of the support play crucial roles in quantizer design.

\medskip

\noindent
\textbf{Future Work.} Several avenues of research emerge from this study. Future investigations could include:
\begin{itemize}
    \item Extending the analysis to mixed distributions formed from more than two overlapping uniform distributions or from non-uniform component distributions.
    \item Studying the asymptotic behavior of optimal quantizers and quantization errors as $n \to \infty$, particularly in relation to quantization dimension.
    \item Applying the proposed algorithm to more general overlapping scenarios in higher dimensions and determining the impact of geometric configurations on quantizer structure.
    \item Exploring applications in signal processing, machine learning, and resource allocation, especially in scenarios where probability densities reflect spatial heterogeneity.
\end{itemize}
These directions will broaden the scope of optimal quantization theory and open new possibilities for both mathematical exploration and practical implementation.

\section*{Declaration}
							
							\noindent
\textbf{Authors' contributions:} Each author contributed equally to this manuscript. All authors have read and agreed to the published version of the manuscript.\\
\\
\noindent \tbf{Funding:}  This research received no external funding.\\
							
							\noindent
							\textbf{Data availability:} No data were used to support this study.\\
							\\
							\noindent
\textbf{Conflicts of interest.} The authors declare no conflict of interest.\\


\begin{thebibliography}{9999}

\bibitem {S}   W.F. Sheppard, \emph{On the calculation of the most probable values of frequency constants for data arranged according to equidistant divisions of a scale}, Proc. London Math. Soc. 29, Part 2, 353-380 (1898).





%
%
%
%
%


\bibitem {GG} A. Gersho and R.M. Gray, \emph{Vector quantization and signal compression}, Kluwer Academy publishers: Boston, 1992.


\bibitem {GKL}  R.M. Gray, J.C. Kieffer and Y. Linde, \emph{Locally optimal block quantizer design}, Information and Control, 45 (1980), pp. 178-198.





%
%

    \bibitem {GN}  R. Gray and D. Neuhoff, \emph{Quantization}, IEEE Trans. Inform. Theory,  44 (1998), pp. 2325-2383.


\bibitem {Z} R. Zam, \emph{Lattice Coding for Signals and Networks: A Structured Coding Approach to Quantization, Modulation, and Multiuser Information Theory}, Cambridge University Press, 2014.


%
%


\bibitem {PR1} M. Pandey and M.K. Roychowdhury, \emph{Constrained quantization for probability distributions}, to appear, Journal of Fractal Geometry, 2025. 
\bibitem {PR2} M. Pandey and M.K. Roychowdhury, \emph{Constrained quantization for the Cantor distribution}, J. Fractal Geom. 11 (2024), no. 3/4, pp. 319-341.
 \bibitem {PR3} M. Pandey and M.K. Roychowdhury, \emph{Conditional constrained and unconstrained quantization for probability distributions}, arXiv:2312:02965 [math.PR].
\bibitem {GL} S. Graf and H. Luschgy, \emph{Foundations of quantization for probability distributions}, Lecture Notes in Mathematics 1730, Springer, Berlin, 2000.

%
%
%


%
\bibitem {R6} M.K. Roychowdhury, \emph{Optimal quantization for mixed distributions}, Real Analysis Exchange, Vol. 46(2), 2021, pp. 451-484.
%

\bibitem {RS} M.K. Roychowdhury and W. Salinas, \emph{Quantization for a mixture of uniform distributions associated with probability vectors}, Uniform Distribution Theory 15 (2020), no. 1, 105-142.

\bibitem {RR} J. Rosenblatt and M.K. Roychowdhury, \emph{Uniform distributions on curves and quantization}, Commun. Korean Math. Soc. 38 (2023), No. 2, pp. 431-450.







\end{thebibliography}
\end{document}